\documentclass[12pt]{article}

\usepackage{latexsym,amsfonts,amsmath,amssymb,epsfig,tabularx,amsthm,dsfont,mathrsfs}
\usepackage{graphicx}
\usepackage{hyperref}
\usepackage{enumerate}
\usepackage{color}

\allowdisplaybreaks

\theoremstyle{plain}

\newtheorem{thm}{Theorem}[section]
\newtheorem{rem}[thm]{Remark}
\newtheorem{alg}{Algorithm}
\newtheorem{lem}[thm]{Lemma}
\newtheorem{prop}[thm]{Proposition}
\newtheorem{cor}[thm]{Corollary}
\newtheorem{example}[thm]{Example}
\newtheorem{ass}[thm]{Assumption}
\newtheorem{defi}[thm]{Definition}

\renewcommand{\d}{{\rm d}}

\newcommand{\norm}[1]{\left\Vert #1 \right\Vert}

\newcommand{\N}{\mathbb{N}}
\newcommand{\E}{\mathbb{E}}
\newcommand{\R}{\mathbb{R}}

\newcommand{\cX}{\mathcal{X}}

\newcommand{\cZ}{\mathcal{Z}}
\newcommand{\cD}{\mathcal{D}}

\DeclareMathOperator{\eps}{\varepsilon}

\DeclareMathOperator{\KL}{\textnormal{KL}}
\DeclareMathOperator*{\argmin}{argmin}
\DeclareMathOperator*{\ecc}{\textnormal{{ecc}}}
\DeclareMathOperator*{\diff}{\textnormal{{diff}}}

\title{
Error bounds for simultaneous Wasserstein contractive adaptive increasingly rare MCMC 
}

\author{Julian Hofstadler\thanks{Department of Mathematical
		Sciences, University of Bath, BA2 7AY, Bath, UK.
		e-mail: \texttt{jh4272@bath.ac.uk}}, Daniel Rudolf\thanks{Chair of Mathematical Data Science, University of Passau, 94032 Passau, Germany. e-mail: \texttt{daniel.rudolf@uni-passau.de}}}
\date{\today}

\begin{document}
\maketitle

\begin{abstract}
We investigate adaptive increasingly rare Markov chain Monte Carlo algorithms and the associated time-average estimator for approximating expectations.
Under a simultaneous Wasserstein contraction assumption on the underlying family of Markov kernels we derive explicit bounds for the mean squared error. 
We illustrate the applicability of our estimate through adaptive stereographic algorithms and Metropolis-Hastings schemes that employ normalizing flows for adaptation. We also consider a generic adaptive algorithm for doubly intractable problems and provide a corresponding cost analysis to achieve a desired precision.
\end{abstract}



\section{Introduction}
In statistics and numerical analysis a challenging task is the computation of integrals of the form
\begin{equation}
\label{eq: pih}
\pi(h) :=	\int_{\cX} h(x)\, \pi({\rm d}x),
\end{equation}
where $(\cX,\mathcal{B}(\cX))$ is a Borel space and $h\colon \cX\to \mathbb{R}$ a measurable $\pi$-integrable function.
Here $\pi$ may be a partially known probability measure with a density that contains an unknown normalization factor as it commonly appears in Bayesian statistics for posterior distributions. Additionally, we either impose or have a natural dependence of $\pi$ to an additional characterizing quantity $z^*\in \cZ$, with $(\cZ,\mathcal{F}_{\cZ})$ being an
auxiliary measurable space. This can or should be beneficially used to address the integration problem. We provide two examples to motivate the latter point of view:

\begin{example}[Normalizing flows]
	\label{ex: norm_flow}
Let $\cX\subseteq \mathbb{R}^d$ and 
$\cZ$ be a parametrized set of $C^1$-diffeomorphisms, i.e., for any $u\in \mathcal{U}\subseteq \mathbb{R}^s$ we have a $C^1$-diffeomorphism $z_u \colon \cX\to\cX$ in $\cZ$. 
Choose an auxiliary probability measure $\widetilde{\pi}$ on $(\cX,\mathcal{B}(\cX))$ such that realizations of $\widetilde{X}\sim \widetilde{\pi}$ can be drawn. Consider $z_u(\widetilde{X})\sim   
\pi_{z_u}:=\widetilde{\pi}\circ z_u^{-1}$ for $z_u\in \cZ$, where $\widetilde{\pi}\circ z_u^{-1}$ denotes the pushforward measure of $\widetilde{\pi}$ by $z_u$.  
The idea is to take advantage of the `best' proxy of $\pi$ through $\pi_{z_u}$s. To this end, define the Kullback-Leibler divergence
\[
	\KL(\pi,\pi_z) := \int_{\cX} \log\left(\frac{\d\pi}{\d \pi_z}(x)\right) \pi({\rm \d}x),
\]
where $\frac{\d\pi}{\d{\pi}_z}$ denotes the density of $\pi$ with respect to (w.r.t.) $\pi_z$. Then, let
\[
z^* := \argmin_{z\in\cZ} \KL(\pi,{\pi}_z)
\]   
be well defined (that is, there exists a unique minimizer). Then, we may use ${\pi}_{z^*}$ as proxy for $\pi$ to construct sampling schemes. However, usually we won't have access to $z^*$, such that it needs to be approximated, cf.~\cite{gabrie2022adaptive}. 	 
\end{example}

\begin{example}[Doubly-intractable distributions]
	\label{ex: doubly}
	Let $\bar{y} \in \cD$ be some measured data, with $(\cD , \mathcal{F}_\cD, \mu)$ being a probability space.
	Assume $\nu$ is a reference probability measure  on $(\cX,\mathcal{B}(\cX))$ serving as prior distribution.
	Consider $E\colon \cD \times \cX \to \mathbb{R}$ as an energy function that induces for $x\in\cX$ a likelihood mapping $y \mapsto \ell(y|x)$ of an energy based model as
	\[
		\ell(y|x) = \frac{\exp(-E(y,x))}{z^*(x)},
	\]
	with $z^*(x) := \int_{\cD} \exp(-E(y,x)) \, \mu({\rm d}y)$. Here $z^*\colon \cX \to (0,\infty)$ is interpreted as  non-evaluable (partition) function. Now the distribution of interest, denoted by $\pi=\pi_{z^*}$, takes the form
	\[
		\pi_{z^*}({\rm d}x) = \frac{\ell(\bar{y}| x) \nu({\rm d}x)}{\int_{\cX} \ell(\bar{y}| \widetilde{x}) \nu({\rm d} \widetilde{x})}
		=
		\frac{\exp(-E(\bar{y},x))}{z^*(x)\, C_{z^*}}\nu({\rm d}x),
	\] 
	with normalizing constant $C_{z^*} := \int_{\cX} \exp(-E(\bar{y},x))\, z^*(x)^{-1} \nu({\rm d}x)$. Here $C_{z^*}$ is unknown and evaluating the function $z^*$ infeasible, which renders the sampling problem w.r.t.~$\pi_{z^*}$ doubly-intractable, compare to \cite{murray2012mcmc}.
If $\cZ$ is a set of mappings $z\colon \cX \to (0,\infty)$ 
	we may approximate $z^*$ by $\widehat{z} \in \cZ$ aiming for $\pi_{\widehat{z}}$, where $\widehat{z}$ just substitutes $z^*$, that is `close' to $\pi_{z^*}$.	
\end{example}

These examples showcase two different conceptional scenarios. In Example~\ref{ex: norm_flow} the dependence of $\pi$ on $z^*$ is imposed and we typically have $\pi_{z^*} \not = \pi$. In contrast to that a natural dependence of $\pi$ on $z^*$ is present in Example~\ref{ex: doubly} and we have $\pi_{z^*} = \pi$.
In both settings it is reasonable to (try to) exploit this underlying  structure in the corresponding sampling problem.

The entire idea is not new and fits well in the classical framework of
adaptive Markov chain Monte Carlo (MCMC). For example, for adaptive Metropolis, cf. \cite{haario2001adaptive}, on the fly a parameter $z^*$, corresponding to the covariance matrix of $\pi$, is learned during simulation. In each iteration a refined proxy of $z^*$ is used to change the proposal in the Metropolis scheme improving the speed of convergence. After this, numerous extensions of adaptations in a Metropolis-Hastings framework, e.g.~in \cite{andrieu2006ergodicity,saksman2010ergodicity, fort2011convergence}, and other Markov chain schemes, e.g.~in \cite{roberts2009examples}, have been studied. 
Also different adaptation strategies, for example via preconditioning are feasible, we refer to \cite{hird2026non} for recent developments. 

We also consider adaptive MCMC, but add restrictions on the schedule of adaptation times where a refined proxy of $z^*$ is used to change the transition mechanisms. 
We follow the adpatively increasingly rare (AIR)
MCMC approach, coined and introduced in \cite{chimisov2018air} where
adaptation, as the name suggests, get less frequent over time. The corresponding AIR process\footnote{We postpone a precise definition to Section \ref{sec: alg}.}, say $(X_n,Z_n)_{n \in \N_0}$, then can be used to approximate $\pi_{z^*}$ by using $(X_n)_{n \in \N_0}$ while $(Z_n)_{n \in \N_0}$ serves as sequence of proxies for $z^*$. A useful observation is that between to adaptation times $(X_n)_{n \in \N_0}$ behaves like a Markov chain, which can be exploited in the theoretical analysis, see \cite{chimisov2018air,hofstadler2026convergenceRates}. Additionally, for carefully designed algorithms, it is still possible to adapt `enough' to significantly improve the sampling mechanism.  

In this work we are primarily concerned with theoretical guarantees of AIR algorithms. Namely, we consider the time-average estimator $\widehat{\pi}_n(h):= \frac{1}{n} \sum_{j=1}^n h(X_j)$ as proxy for $\pi(h)$, cf.~\eqref{eq: pih}, and the goal is to quantify the mean squared error of $\widehat{\pi}_n(h)$. Our main results, Theorem \ref{thm:mse_error_general} and Corollaries \ref{cor:bias_free_bnd_ecc} and \ref{cor:error_decomp_bounded_eccentricities}, show that under a simultaneous Wasserstein contraction assumption we have 
\[
\E \left[\left \vert \widehat{\pi}_n(h) - \pi_{z^*}(h)  \right\vert^2  \right] \leq \frac{C_h}{n} \left( 1 + B_n\right), 
\]
where $B_n$ is a bias term that only depends on our ability to approximate $\pi_{z^*}$ and the dependence on $h$ in $C_h$ is explicit. If all involved kernels are $\pi_{z^*}$-invariant, then $B_n=0$, see Corollary \ref{cor:bias_free_bnd_ecc}. If this is not the case, but still the bias term vanishes, then Corollary~\ref{cor:error_decomp_bounded_eccentricities} allows us to get 
refined
bounds for the mean squared error.
Furthermore, we apply our results to the following examples,
\begin{itemize}
\item[1)] adaptive stereographic MCMC, see Section~\ref{subsec:adaptive_stereo}, 
\item[2)] adaptive Metropolis via normalizing flows, see Section \ref{subsec:normalising_flow_AMH}, and
\item[3)] doubly intractable problems, see Section \ref{sec:doubly_intractable}.
\end{itemize}
For all examples we are able to deduce explicit error bounds for the mean squared error of $\widehat{\pi}_n(h)$ for suitably designed AIR algorithms. Additionally, for the doubly intractable setting we carry out a cost analysis based on our bounds. To the best of our knowledge this was not done before for this class of problems in the context of adaptive MCMC. 

Let us comment on how our results fit into the existing literature. AIR algorithms were introduced and studied in \cite{chimisov2018air} and later also in \cite{hofstadler2026convergenceRates,laitinen2024invitation}. Mean squared error bounds were obtained in \cite{chimisov2018air} for the case of simultaneous uniformly and $V$-uniformly ergodic algorithms, additionally assuming that $\pi_z = \pi$ for any $z \in \mathcal{Z}$. In the present work we consider a more general simultaneous Wasserstein contraction assumption (see Definition \ref{defi:uniform_Wass_contr}) and allow for $(\pi_z)_{z \in \mathcal{Z}}$ to vary. Similar conditions were used recently in \cite{hofstadler2026convergenceRates} to study almost sure convergence rates for AIR methods. 

Our analysis is based on a martingale decomposition technique that relies on solutions of Poisson's equation. This is a standard tool in the adaptive MCMC literature, see e.g.~\cite{andrieu2006ergodicity,saksman2010ergodicity,atchade2013bayesian,laitinen2024invitation,hofstadler2026convergenceRates}, particularly to study almost sure convergence and central limit theorems, see also \cite{atchade2010limit,atchade2012limit}. The present work is concerned with pre-asymptotic error bounds, which complement these existing results. 

It is worth mentioning that mean squared error bounds are well studied for classical, i.e.~non-adaptive, MCMC methods, we refer to \cite{Joulin_Ollivier,Ru12,latuszynski2015nonasymptotic,hofstadler2025optimal}. The paper \cite{Joulin_Ollivier} uses a Wasserstein contraction assumption, equivalently stated in terms of Ricci curvature there, to deduce such bounds. Our error estimates can also be considered as generalization of their framework to the adaptive setting.

The rest of this paper is organized as follows. In Section~\ref{sec:setting} we specify the setting. We provide  necessary background on Wasserstein distances, give a precise definition of the a generic AIR algorithm and collect auxiliary results. Section~\ref{sec:main_results} contains our main statements about the mean squared error and examples that illustrate the applicability of our contribution.
 The doubly intractable setting is investigated in Section \ref{sec:doubly_intractable}, where we establish error bounds and carry out a cost analysis. Finally, the appendix contains technical proofs.

\section{Preliminaries}\label{sec:setting}

Let $(\Omega,\mathcal{F},\mathbb{P})$ be a sufficiently rich probability space that serves as domain for all appearing random variables. 
Let $\cX$ be a Polish space and let $\mathcal{B}(\cX)$ be the corresponding Borel $\sigma$-algebra.
Let $(\mathcal{Z},\mathcal{F}_{\mathcal{Z}})$ be an auxiliary measurable space and
$\Pi \colon \mathcal{Z}\times \mathcal{B}(\cX)  \to [0,1]$ 
be a transition kernel that satisfies $\Pi(z,\cdot)=\pi_z$ for all $z\in\cZ$, where $(\pi_z)_{z\in\cZ}$ is a sequence of probability measures on $(\cX,\mathcal{B}(\cX))$.
By $\delta_{y}$ with $y$ being in a measurable space (e.g.~$\cX$, $\mathcal{Z}$ or $\mathcal{D}$) we denote the corresponding Dirac measure.

If $P\colon \cX \times \mathcal{B}(\cX) \to [0,1]$ is a generic Markov kernel, and $f\colon \cX \to \R$ we write
\[
Pf(x) = \int_\cX f(y) P(x, \d y),
\]
whenever the integral is well-defined. Similarly, we write $\mu(f) = \int_\cX f(x) \mu(\d x)$ if $\mu$ is a probability measure on $(\cX, \mathcal{B}(\cX))$ assuming the integral makes sense. 
Additionally, with $A \in \mathcal{B}(\cX)$, we use the standard notation
\[
\mu P(A) = \int_{\cX} P(x, A) \mu(\d x).
\]

\subsection{Adaptive increasingly rare MCMC}
\label{sec: alg}
The goal of the present paper is to study sequences $(X_n, Z_n)_{n \in \N_0}$ of $(\cX , \mathcal{Z})$-valued random variables corresponding to AIR MCMC schemes. 
For providing an algorithmic description that determines $(X_n)_{n\in\mathbb{N}_0}$ in terms of an auxiliary $\cZ$-valued sequence of random variables $(Z_n)_{n\in\mathbb{N}_0}$ we require
\begin{itemize}
		\item a probability measure $p_0$ on $(\cX,\mathcal{B}(\cX))$ that serves as initial distribution, i.e., $p_0(A) = \mathbb{P}(X_0\in A)$ for all $A\in\mathcal{B}(\cX)$;
	\item a transition kernel $Q_0 \colon \mathcal{X}\times \mathcal{F}_{\mathcal{Z}} \to [0,1]$  which specifies for $B\in\mathcal{F}_{\mathcal{Z}}$ and $x\in \mathcal{X}$ that
	\begin{equation}
	\label{eq: law_Z_0}
	Q_0(x,B) =  \mathbb{P}(Z_0\in B\mid X_0=x );
	\end{equation}
	\item a transition kernel $K\colon (\mathcal{Z}\times \cX) \times \mathcal{B}(\cX)  \to [0,1]$ which specifies for any $(z,x)\in\mathcal{Z}\times \cX$, $n\in\mathbb{N}_0$ and $A\in\mathcal{B}(\cX)$ that 
	\begin{equation} \label{eq: K}
			K(z,x,A) = \mathbb{P}(X_{n+1}\in A\mid X_n=x, Z_n=z);
	\end{equation}
	\item for any $n\in\N$ such that $n={j(j+1)}/{2}$ for some $j \in \N$ a transition kernel
	$
	Q_n \colon( \cX^{n +1}\times \mathcal{Z}^{n}) \times \mathcal{F}_{\mathcal{Z}} \to [0,1]
	$ 
	which specifies for any $(\widetilde{x},\widetilde{z})\in \cX^{n +1}\times \mathcal{Z}^{n}$
	and $B\in\mathcal{F}_{\cZ}$ that
	\[
		Q_n(\widetilde{x},\widetilde{z},B) = \mathbb{P}(Z_{n}\in B\mid (X_0,\dots,X_{n})=\widetilde{x}, (Z_0,\dots,Z_{n-1})=\widetilde{z} ).
	\]
\end{itemize}
Note that, the conditional distribution of $Z_0$ given $X_0$ is specified in \eqref{eq: law_Z_0} and the conditional distribution of $X_1$ given $X_0,Z_0$ in \eqref{eq: K}.
Now we are able to provide the transition mechanism, see Algorithm~\ref{alg:generic_scheme}, describing how to get $X_{n+1}$ given $X_0,\dots,X_n$ and $Z_0,\dots,Z_{n-1}$ for $n\in\mathbb{N}$. 
\begin{alg}
	\label{alg:generic_scheme}
	For $n\in\mathbb{N}$ given $(X_0,\dots,X_n)=\widetilde{x}\in\cX^{n+1}$ with $X_n=x_n\in\cX$ and $(Z_0,\dots,Z_{n-1})=\widetilde{z}\in\cZ^{n}$ we sample $X_{n+1}$ by performing:
	\begin{enumerate}
		\item \textbf{If} $n= j(j+1)/2$ for some $j\in \N$,
		
		\quad \textbf{then} sample $Z_{n}\sim Q_n(\widetilde{x}, \widetilde{z}, \cdot)$, call the result $z_{n}\in \mathcal{Z}$;
		
		\textbf{Else} set $z_{n} = z_{n-1}$, i.e., $Z_n\sim \delta_{z_{n-1}}$.
		\item Sample $X_{n+1}\sim K(z_{n},x_{n},\cdot)$.
	\end{enumerate}
\end{alg}

\begin{rem}
More generally, in \cite{chimisov2018air,hofstadler2026convergenceRates} adaptation times of the form $t_j^{(\beta)} = \sum_{i=1}^j \lceil i^\beta \rceil $, with $j\in \N$, have been considered. The results developed there indicate, theoretically and numerically, that $\beta \in [1,2]$ is particularly appealing. 
Motivated by this and for notational simplicity, in Algorithm~\ref{alg:generic_scheme} we have chosen $\beta=1$, leading to $t_j^{(1)} = j(j+1)/2$. 
However, also other choices of $(t_j)_{j\in\mathbb{N}}$, even randomized ones, as suggested in \cite{laitinen2024invitation}, may be eligible.
\end{rem}

\subsection{Wasserstein distance and contraction}

Let $\rho \colon \cX\times\cX \to [0,\infty)$ be a lower semi-continuous (w.r.t.~the product topology on $\cX$) metric. For probability measures $\mu_1,\mu_2$ on $(\cX,\mathcal{B}(\cX))$ define 
\[
W(\mu_1,\mu_2) := \inf_{\xi\in C(\mu_1,\mu_2)} \int_{\cX^2} \rho(x_1,x_2) \xi({\rm d}x_1,{\rm d}x_2), 
\]
where $C(\mu_1,\mu_2)$ is the set of couplings of $\mu_1,\mu_2$. We refer to it as \emph{Wasserstein distance}.
By the well-known Kantorovich-Rubinstein duality formula, cf.~\cite{villani2009optimal}, we have the following identity
\begin{equation}
	\label{eq: Kant_Rub_vs}
	W(\mu_1,\mu_2) = \sup_{\Vert h \Vert_\rho \leq 1} \left\vert \int_{\cX} h(x) \,\mu_1(\d x) - \int_{\cX} h(x) \,\mu_2(\d x)\right \vert,
\end{equation}
where $\norm{h}_\rho = \sup_{x,y \in \cX \, ; \, x \neq y} \frac{\vert h(x) - h(y) \vert }{\rho(x,y)}$ is the Lipschitz semi-norm of $h \colon \cX \to \R$ w.r.t.~the metric $\rho$.

For the rest of the section fix $z\in \cZ$. 
For $(x,A) \in (\cX,\mathcal{B}(\cX))$ set $K_z (x, A) := K(z,x, A)$. Note that $K_z$ is a transition kernel on $(\cX, \mathcal{B}(\cX))$ and 
recall that $\Pi(z,\cdot)=\pi_z$ is a probability measure on $(\cX,\mathcal{B}(\cX))$. 

For $k \in \mathbb{N}$ we call 
\[
\tau(K_z^k) := \sup_{x_1,x_2\in \cX\, ; \, x_1\not=x_2}\frac{W(K_z^k(x_1,\cdot),K^k_z(x_2,\cdot))}{\rho(x_1,x_2)}
\]
the
Wasserstein contraction coefficient of $K_z^{k}$, with $K_z^{k}$ being the product of $K_z$ with itself $k$-times.
By \cite[Proposition 14.3 and Proposition 14.4]{Do96} the contraction coefficient satisfies
\begin{itemize}
	\item a submultiplicativity property, i.e.~$\tau (K_{z_1} K_{z_2}) \leq \tau(K_{z_1}) \tau(K_{z_2})$ for any $z_1,z_2\in \mathcal{Z}$; and
	\item a contraction property, i.e.~$W(\mu_1 K^k_z,\mu_2 K^k_z) \leq \tau(K^k_z)\, W(\mu_1,\mu_2)$ for any $k\in\mathbb{N}$ and probability measures $\mu_1,\mu_2$ on $(\cX,\mathcal{B}(\cX))$. 
\end{itemize}

Slightly adopting the notation of \cite{hofstadler2026Poisson} we define the following. 

\begin{defi}\label{defi_Wasserstein_contractive_kernel}
The transition kernel $K_z$ is called \textit{Wasserstein contractive} if it admits $\pi_z$ as invariant distribution, i.e.~$\pi_z K_z = \pi_z$, satisfies $\tau(K_z^\ell) <1 $ for some $\ell \in \N$ and we have that $\tau(K_z) < \infty$.
\end{defi}

By submultiplicativity and contractivity we have the following:  
If $K_z$ is Wasserstein contractive, then for any probability distribution $\mu$ on $(\cX , \mathcal{B}(\cX))$ we have $W(\mu K_z^k, \pi_z) \le C \tau^k$ for a constant $C < \infty$.
 
\subsection{Poisson's equation} 

To prove our main results we intend to use a martingale approximation technique based on solutions of Poisson's equation. 
In this section we present the required background and preliminaries about Poisson's equation.

For the rest of this section let $z \in \mathcal{Z}$ be fixed. Given a $\pi_z$-integrable function $f \colon \cX \to \R$, referred to as \textit{forcing function}, we call 
\begin{equation}\label{equ:Poisson_equation}
u_z(x) - K_zu_z(x) = f(x) - \pi_z(f),
\end{equation}
the \textit{Poisson equation} (with forcing function $f$). If $u_z \colon \cX \to \R$ satisfies \eqref{equ:Poisson_equation}, then we say it solves Poisson's equation. 

In the spirit of \cite{Joulin_Ollivier} we define the \textit{eccentricity}, and \textit{coarse diffusion coefficient} of the probability space $(\cX,\mathcal{B}(\cX),\pi_z)$, respectively, as 
\[
	{\ecc}_z(x) :=\int_{\cX} \rho(x,x') \pi_z({\rm d} x'),
\]
and 
\[
	\diff(z,x):=\int_\cX \int_\cX \rho(x_1,x_2)^2 K_z(x,{\rm d}x_1) K_z(x,{\rm d}x_2).
\]

The following result guarantees existence and regularity of solutions to Poisson's equation for Wasserstein contractive $K_z$, see also \cite{hofstadler2026Poisson}. 

\begin{prop}\label{prop:solution_poisson_equation}
	Let $z\in\cZ$ and $f \colon \cX \to \mathbb{R}$ be $\pi_z$-integrable with $\Vert f \Vert_{\rho}< \infty$. Assume that there exists an $\widetilde{x}\in\cX$ such that ${\ecc}_{z}(\widetilde{x})<\infty$. If $K_z$ is a Wasserstein contractive kernel, then, the mapping 
	\[
	x\mapsto u_z(x) := \sum_{k=0}^\infty \left( (K^k_z f)(x) - \pi_z(f) \right), \quad x\in\cX,
	\]
is a well-defined solution of Poisson's equation w.r.t.~forcing function $f$, and for  $\ell = \min\{k \in \N \colon \tau(K_z^k) < 1\} < \infty$ satisfies
	\begin{enumerate}[(a)]
		\item  Lipschitz boundedness, that is, $\Vert u_z \Vert_{\rho} \leq\frac{\tau(K_z)^\ell \Vert f \Vert_{\rho} }{1-\tau(K_z^\ell)}$; and
		\item for any $x\in\cX$ that 
		$\vert u_z(x) \vert\leq  \frac{\Vert f \Vert_{\rho} \tau(K_z)^\ell }{1-\tau(K_z^\ell)}\, {\ecc}_z(x)$ and $u_z(x)\in\mathbb{R}$.
	\end{enumerate}	
\end{prop}

\begin{proof}
By \cite[Theorem 3.4]{hofstadler2026Poisson}, $u_z$ solves Poisson's equation w.r.t.~forcing function $f$ and satisfies property (a). Property (b) follows from the proof of \cite[Theorem 3.5]{hofstadler2026Poisson}, see in particular Equation (3) there. 
\end{proof}

\begin{rem}
	In the proofs of the next sections expectations of the form $ \E [{\ecc}_Z(X)]$ or $\E[ \int_\cX \rho(X, x') K(Z,X, \d x')]$, with $\mathcal{Z}$-valued random variable $Z$ and $\cX$-valued random variable $X$
	appear. 
	Let us emphasize that ${\ecc}_Z(X)$ and $\int_\cX \rho(X, x') K(Z, X, \d x')$ are indeed well defined (non-negative) random variables, compare for example~\cite[Appendix B]{hofstadler2026convergenceRates}.
\end{rem}

\subsection{Assumptions on the AIR algorithm} 
In this section we formulate our regularity condition on the AIR algorithm which we use to carry out our analysis later. Our assumptions are twofold. Firstly, we require a simultaneous Wasserstein contraction for $(K_z)_{ z \in \mathcal{Z}}$. Secondly, we assume that the metric $\rho$ used in $W$ is sufficiently regular when combined with $(X_n, Z_n)_{n \in \N_0}$.

\begin{defi}
\label{defi:uniform_Wass_contr}
We say that the family of kernels $(K_z)_{z \in \mathcal{Z}}$, used in Algorithm~\ref{alg:generic_scheme}, is \textit{simultaneous Wasserstein contractive} if there exist $C \in (0, \infty)$, $\tau\in [0,1)$ and $\ell \in \N$ such that 
	for every $z\in\cZ$ the invariant distribution of $K_z$ is $\pi_z$,
\[	
\tau(K_z) \leq C \qquad \text{and} \qquad \tau(K_z^\ell) \leq \tau.
\]
\end{defi} 

\begin{rem}
We adopt the convention that the number $\ell\in \N$ appearing in Definition \ref{defi:uniform_Wass_contr} is always the smallest such number. 
\end{rem}

\begin{rem}
Requiring the family $(K_z)_{z \in \mathcal{Z}}$ to be Wasserstein contractive with constants independent of $z\in \mathcal{Z}$ is certainly restrictive. However, without adding other assumptions, e.g.~on $(Z_n)_{n \in \N_0}$, weakening this 
condition
is not possible, cf.~\cite[Appendix C]{hofstadler2026convergenceRates}.  
Additionally, by appropriately choosing $\rho$, it was shown in \cite{hofstadler2026convergenceRates} that the \textit{simultaneous uniformly ergodic} and \textit{simultaneous $V$-uniformly ergodic} settings are covered, see also \cite{saksman2010ergodicity,chimisov2018air,bell2024adaptivestereographicmcmc,laitinen2024invitation} as well as Sections~\ref{subsec:first_examples} and \ref{sec:doubly_intractable} for examples.
\end{rem}

In addition to simultaneous Wasserstein contractivity, we also require some sort of regularity connecting the metric $\rho$ with the process $(X_n, Z_n)_{n \in \N_0}$. 

\begin{defi}\label{defi:metric_regularity}
Let $(X_n, Z_n)_{n \in \N_0}$ be a process corresponding to Algorithm~\ref{alg:generic_scheme}. We call $(X_n , Z_n)_{n \in \N_0}$ \textit{metric regular}, if 
\[
\sup_{j,k \in \N_0} \E\left[ {\ecc}_{Z_j}(X_k) \right] < \infty \qquad \text{and} \qquad \sup_{j,k \in \N_0} \E[{\diff}_{Z_j}(X_k)^2] < \infty.
\]
\end{defi}

\begin{rem} \label{rem:metric_bnd}
For any bounded metric $\rho$, that is, $\sup_{x_1,x_2 \in\mathcal{X}} \rho(x_1,x_2)<\infty$, the process $(X_n, Z_n)_{n \in \N_0}$ is metric regular by definition of the eccentricity and coarse diffusion coefficient.
\end{rem}

The following result provides a sufficient condition for $(X_n, Z_n)_{n \in \N_0}$ to be metric regular in terms of a simultaneous Lyapunov function. 

\begin{prop}\label{prop_uniform_bound_eccentr_and_diffusion}
	Let $V\colon\cX \to [1,\infty)$ and $\alpha\in [0,1)$, $L>0$ such that
	for any $(z,x)\in\mathcal{Z}\times \cX$ we have $K_z V(x)\leq \alpha V(x) + L$, i.e.~$V$ is a uniform Lyapunov function for $(K_z)_{z \in \mathcal{Z}}$. 
	Assume $\mathbb{E}[V(X_0)]<\infty$ and for $\widetilde{L}>0$ let
	\[
	\rho(x_1,x_2) \leq \widetilde{L} \left(V(x_1)^{1/2}+V(x_2)^{1/2}\right),
	\] 
	for any $x_1,x_2\in \cX$.
	Then, $(X_n , Z_n)_{n \in \N_0}$ is metric regular. 
\end{prop}
\begin{proof} By \cite[Proposition 4.24]{hairer2006ergodic} for any $z \in \mathcal{Z}$ we have $\pi_z(V) \leq \frac{L}{1-\alpha}$. 
As a consequence of the Lyapunov function property 
	\begin{align}
		\notag
		\E [ V(X_k)] &= \E \left[\int_\cX V(x) K(Z_{k-1}, X_{k-1}, \d x) \right] \leq \E [\alpha V(X_{k-1}) + L] \\
		& \label{al: Lyap_est}
		\leq \dots  
		\leq \alpha^k\E [ V(X_0) ] + L\sum_{j=0}^{k} \alpha^j
		\leq \E [ V(X_0) ] + \frac{L}{1-\alpha}.
	\end{align}
	Using \eqref{al: Lyap_est} we obtain uniformly for any $k,\ell\in\N$ that
	\begin{align*}
		\mathbb{E} [ {\ecc}_{Z_\ell}(X_k)^2] & \leq \widetilde{L}^2\,\E \left[\int_{\cX} (V(X_k)^{1/2} + V(x)^{1/2})^2\pi_{Z_\ell}(\d x)\right]\\
		&\leq  2\widetilde{L}^2\, \E [ V(X_k) + \pi_{Z_\ell}(V)]
		\leq 2\widetilde{L}^2\,\E [ V(X_0) ] + \frac{4L\widetilde{L}^2}{1-\alpha} < \infty.
	\end{align*}
	Furthermore, by 
	\eqref{al: Lyap_est} and the imposed assumptions, for any $k\in\N$ we get
	\begin{align*}
		&\E [\diff(Z_k,X_k)]
		= \E \Big[ \int_{\cX} \int_{\cX} \rho(x_1,x_2)^2 K_{Z_k}(X_k,\d x_1) K_{Z_k}(X_k,\d x_2)\Big]\\
		& \leq \E \Big[ \int_{\cX} \int_{\cX} \widetilde{L}^2 (V(x_1)^{1/2}+V(x_2)^{1/2})^2 K_{Z_k}(X_k,\d x_1) K_{Z_k}(X_k,\d x_2)\Big]\\
		& \leq 2 \widetilde{L}^2 \E \Big[ \int_{\cX} \int_{\cX}  (V(x_1)+V(x_2)) K_{Z_k}(X_k,\d x_1) K_{Z_k}(X_k,\d x_2)\Big]\\
		& \leq 4 \widetilde{L}^2 \E [K_{Z_k} V(X_k)] = 4 \widetilde{L}^2 \E [V(X_{k+1})] \underset{\eqref{al: Lyap_est}}{<}\infty.
		\qedhere
	\end{align*}
\end{proof}

\section{Mean squared error bounds}\label{sec:main_results}

Now we formulate and prove our main result, which is a mean squared error bound for $\widehat{\pi}_n(h)$.
For this, throughout this section, let $(X_n, Z_n)_{n\in\mathbb{N}_0}$ be an AIR process as determined by Algorithm~\ref{alg:generic_scheme} specified by the transition kernels $K\colon (\mathcal{Z}\times \cX) \times \mathcal{B}(\cX)  \to [0,1]$ and $Q_n \colon( \cX^{n +1}\times \mathcal{Z}^{n}) \times \mathcal{F}_{\mathcal{Z}} \to [0,1]$ for suitable $n \in \N$, see Section~\ref{sec: alg}. 
Additionally, the transition kernel $\Pi\colon \cZ\times \mathcal{B}(\cX) \to [0,1]$ serves as our target distribution sequence $(\pi_z)_{z\in\cZ}$ by $\pi_z := \Pi(z,\cdot)$.

We rely on a simultaneous Wasserstein contractivity condition and assume that there is a single particular element $z^* \in \mathcal{Z}$, such that $\pi_{z^*}$ is the desired target distribution, i.e. $\pi_{z^*} = \pi$, with $\pi$ as in \eqref{eq: pih}, or at least $\pi_{z^*}$ being `close' to $\pi$.
Intuitively, $(Z_n)_{n \in \N_0}$ is designed to approximate $z^*$, however, no particular assumptions concerning that are imposed.

For fixed $n\in\mathbb{N}$ define the number
\begin{equation}
	\label{equ:m(n)_definition}
	m:=m(n):= \inf\left\{ k\in\mathbb{N}\colon \frac{k(k+1)}{2}\leq n < \frac{(k+1)(k+2)}{2} \right\}.
\end{equation}
Recall, that the adapatation times, i.e.~the update indices of $Z_n$ in Algorithm~\ref{alg:generic_scheme}, are $j(j+1)$ for $j \in \N$.
Thus, for any $n\in \N$ we have that $m(n)$ is the number of adaptations that were performed to compute $(X_k, Z_k)_{k=0}^n$.
We note that
	\begin{equation}
		\label{equ:m(n)_estimate}
	\sqrt{n/2} \leq m(n)\leq 2\sqrt{n}.
	\end{equation}

\subsection{Main results}\label{subsec:main_results}
 The mean squared error of the time-average estimator based on the AIR MCMC algorithm is stated in Theorem \ref{thm:mse_error_general}. For illustrative purposes additionally, we formulate two consequences, see Corollaries \ref{cor:bias_free_bnd_ecc} and \ref{cor:error_decomp_bounded_eccentricities}, and discuss these results. The proof of Theorem \ref{thm:mse_error_general} is presented in Appendix~\ref{app:proof_main_results}.

We use the following notation: 
For a simultaneously Wasserstein contractive family $(K_z)_{z \in \mathcal{Z}}$, set $\Gamma = \frac{C^\ell}{1-\tau} < \infty$, with $C,\tau$ and $\ell$ as in Definition~\ref{defi:uniform_Wass_contr}. 

\begin{thm}\label{thm:mse_error_general}
	Let $(X_n, Z_n)_{n\in\mathbb{N}_0}$  be determined by Algorithm~\ref{alg:generic_scheme} and $h\colon \cX \to \mathbb{R}$ be $\pi_{z}$-integrable for $z \in \mathcal{Z} \cup \{z^*\}$ with $\Vert h \Vert_\rho<\infty$. Assume $(K_z)_{z \in \mathcal{Z}}$ is simultaneous Wasserstein contractive and $(X_n, Z_n)_{n \in \N_0}$ is metric regular.
	Then, for $n\in\mathbb{N}$, with $m=m(n)$ as in \eqref{equ:m(n)_definition}, 
	we have
	\[
	\mathbb{E}\vert \widehat{\pi}_n(h) - \pi_{z^*}(h) \vert^2 
	\leq \frac{3\norm{h}_\rho^2}{n}\left( \Gamma^2\Lambda + 4\Gamma^2\kappa + B(m) \right) ,
	\]
	with $\Lambda := \sup_{j,k \in \N_0} \E [{\diff_{Z_j}(X_k)}] < \infty$ and $\kappa:=\sup_{j,k\in\mathbb{N}_0} \mathbb{E}({\ecc}_{Z_j}(X_k)^2) <\infty$ as well as $B(m) := \sum_{k=1}^{m}
	(k+1) \sup_{\Vert g \Vert_\rho\leq 1} \mathbb{E}\vert \pi_{Z_{t_k}}(g)-\pi_{z^*}(g) \vert^2$. 
\end{thm}

In Theorem \ref{thm:mse_error_general} there appears a term $B(m)$, which is interpreted as \textit{bias term} and depends on the sequence $(Z_n)_{n \in \N_0}$ that is used to approximate $z^*$. We provide two further results, where, under additional assumptions, we are able to provide more explicit bounds. 
The first result covers the case where we have $\pi_z = \pi$ for any $z \in \mathcal{Z}$ 
with vanishing bias term.

\begin{cor}\label{cor:bias_free_bnd_ecc}
	Let the assumptions and the setting be as in Theorem~\ref{thm:mse_error_general}.
	Additionally, assume that $\pi_z= \pi_{z^*}$ for any $z\in\mathcal{Z}$ .
	Then, for any $n \in \N$, 
	\[
	\mathbb{E}\vert \widehat{\pi}_n(h) - \pi_{z^*}(h) \vert^2 
	\leq \frac{3\norm{h}_\rho^2 \Gamma^2}{n}\left( \Lambda + 4\kappa \right).
	\]	
\end{cor}
We note that the estimates of Corollary \ref{cor:bias_free_bnd_ecc} match the rate one would have for classical MCMC, see \cite{Joulin_Ollivier,Ru12, latuszynski2015nonasymptotic,hofstadler2025optimal}.

The second corollary keeps the bias term and offers bounds of it under suitable proxy assumptions.

\begin{cor}\label{cor:error_decomp_bounded_eccentricities}
	Let the assumptions and the setting be as in Theorem~\ref{thm:mse_error_general}.
	Additionally, assume $\sup_{\Vert g \Vert_\rho\leq 1} \mathbb{E}\vert \pi_{Z_{t_k}}(g)-\pi_{z^*}(g) \vert^2 \leq b/ (k+1)^{1+\alpha} $
	for an absolute constant $b<\infty$ and $\alpha\geq 0$.
	Then, for any $n\in \N$ with $n\geq3$,
	\[
	\mathbb{E}\vert \widehat{\pi}_n(h) - \pi_{z^*}(h) \vert^2  \leq
	\begin{cases}
		\begin{aligned}
&\frac{3\norm{h}_\rho^2C_1}{n}  
+ \frac{6\norm{h}_\rho^2b}{(1-\alpha)\,n^{\frac{1}{2} + \frac{\alpha}{2}}}
&\qquad &\alpha\in [0,1)\\ 
&\frac{3\norm{h}_\rho^2 C_1}{n} +	\frac{
	3
	\norm{h}_\rho^2b\log(3\sqrt{n})}{n} & &\alpha=1\\ 
&\frac{3\norm{h}_\rho^2 C_1}{n}	+ \frac{3\norm{h}_\rho^2b
	}{
	(1-\alpha)
	n} & &\alpha\in (1,\infty),	
\end{aligned}
\end{cases}
	\]
	where 
	$C_1 = \Gamma^2\Lambda + 4\Gamma^2\kappa$.
\end{cor}

\begin{proof}
We begin with the first case, i.e.~$\alpha \in [0,1)$. Note that by definition of $B(m)$ and \eqref{equ:m(n)_estimate} we obtain 
\[
B(m) \leq \sum_{k=1}^m \frac{b}{(k+1)^\alpha} 
	 \leq \frac{b}{1-\alpha}\big[(m+1)^{1-\alpha}-1\big]
	 \leq \frac{3\,b}{1-\alpha} n^{\frac{1}{2}-\frac{\alpha}{2}}. 
\]
Combining this with the estimate from Theorem \ref{thm:mse_error_general} implies the claim. 
In the case $\alpha=1$, we have $B(m) \leq b\log(m+1)$, such that by \eqref{equ:m(n)_estimate} (and $n\geq3$) the conclusion of this part follows. In the final case $\alpha\in(1,\infty)$ we have
\[
B(m) \leq \sum_{k=1}^m \frac{b}{(k+1)^\alpha} \leq \int_0^\infty \frac{b}{(r+1)^{\alpha}} {\,\rm d}r = \frac{b}{1-\alpha}.
\]  
which also by \eqref{equ:m(n)_estimate} yields the bound. 
\end{proof}

\begin{rem}\label{rem:alternative_T_1_bound}
	By the Kantorovich-Rubinstein duality the bias part in Theorem \ref{thm:mse_error_general},
i.e.~$ \sup_{\Vert g \Vert_\rho\leq 1} \mathbb{E}\vert \pi_{Z_{t_\ell}}(g)-\pi_{z^*}(g) \vert^2$, can be expressed as
	\[
	\mathbb{E}\big[
	\sup_{\Vert g \Vert_\rho\leq 1}
	\vert \pi_{Z_{t_\ell}}(g)-\pi_{z^*}(g) \vert^2\big]
=
	\mathbb{E}
	\big[ W(\pi_{Z_{t_\ell}},\pi_{z^*})^2
	\big].
	\]
	Note that the Wasserstein expression in the latter expectation is indeed a measurable function, see e.g.~\cite[Theorem~1]{Zh00}. \end{rem}

\subsection{Application to classical adaptive MCMC}\label{subsec:first_examples}

Next we apply our results to two examples from the adaptive MCMC literature. 
Here `classical' refers to the fact that there is no bias term, i.e. $\pi_z=\pi$ for all $z\in\mathcal{Z}$, meaning each transition kernel $K_z$, with $z \in \mathcal{Z}$, has $\pi$ as its invariant distribution. Furthermore, the underlying metric $\rho$ is the trivial one, i.e.~we rely on what is usually coined as `uniform ergodicity' setting.
Adaptation is incorporated by updating, in a suitable sense, the underlying transition kernel.

\subsubsection{Adaptive stereographic MCMC}\label{subsec:adaptive_stereo}
We consider the stereographic MCMC approach, introduced in \cite{yang2022stereographicMCMC}, and further developed in \cite{bell2024adaptivestereographicmcmc} to the adaptive setting.

Let $\cX = \R^d$ and $\mathcal{B}(\cX)$ be the corresponding Borel sets.
Assume that the target distribution $\pi$ on $(\cX, \mathcal{B}(\cX))$ admits a strictly positive and continuous density $p \colon \cX \to (0, \infty)$. Furthermore, throughout this section, let \cite[Assumption 3.3]{bell2024adaptivestereographicmcmc} be satisfied, that is, we impose that
\[
\limsup_{\norm{x} \to \infty} \; p(x) \left( \norm{x}^2 + 1\right)^d < \infty,
\]
with $\norm{x}$ denoting the Euclidean norm of a vector $x \in \R^d$.

The core idea of the stereographic approach is to beneficially exploit the preconditioned stereographic projection. It bijectively transforms $\pi$ to the $d$-dimensional unit sphere $\mathbb{S}^{d}$. 
The geometric structure and compactness of $\mathbb{S}^d$ allow `efficient approximate sampling' w.r.t. the transformed target distribution there, followed by projecting the output back to $\mathbb{R}^d$.
Adaptivity is incorporated by updating the preconditioner and different algorithms on the sphere may be used, e.g.~Metropolis-Hastings, or slice samplers, cf.~\cite{bell2024adaptivestereographicmcmc}. 

Let $\rho(x,y)$ be the trivial metric, that is $\rho(x,y) = 1$ if $x \neq y$, and $0$ otherwise, with $x,y \in\R^d$.
It is well known that in this case $W$ coincides with the total variation distance (up to a multiplicaitive constant), check for instance \cite[Remark~2.1]{rudolf2018pertubation}.
In this example, the space $\mathcal{Z}$ is set to be a class of suitable positive definite matrices times a class of eligible mean vectors, that together determine a preconditioner, compare to \cite[Section 2]{bell2024adaptivestereographicmcmc}.
Here $\pi$ is assumed to be independent of $z\in\mathcal{Z}$, that is, $\pi = \pi_z$ for any $z \in \mathcal{Z}$, although on $\mathbb{S}^d$ different `transformed targets' are used.

\begin{prop}\label{prop:adaptive_stereographic_mse}
Let $(X_n, Z_n)_{n \in \N_0}$ be a sequence of random variables as specified in Algorithm \ref{alg:generic_scheme}, based either on an AIR stereographic random walk or 
on AIR stereographic spherical slice sampling, see \cite[Section 2]{bell2024adaptivestereographicmcmc}.
Then, there exists a constant $\widetilde{C} < \infty$, such that for any bounded measurable $h \colon \cX \to \R$ and $n \in \N$ we have 
\[
\E \left\vert \widehat{\pi}_n(h) - \pi(h)  \right\vert^2 \leq \frac{\widetilde{C} \norm{h}_{\infty}}{n},
\]
where $\norm{h}_{\infty} = \sup_{x \in \cX} \vert h (x) \vert < \infty$.
\end{prop}
\begin{proof}
By Lemmas 4.1 and 4.2 of \cite{bell2024adaptivestereographicmcmc} the corresponding families of kernels $(K_z)_{z \in \N_0}$ are simultaneous uniformly ergodic. Hence, by the same arguments 
as
in \cite[Section 4.1]{hofstadler2026convergenceRates}, $(K_z)_{z \in \mathcal{Z}}$ satisfies the simultaneous Wasserstein contraction. Boundedness of $\rho$ implies that $(X_n, Z_n)_{n \in \N_0}$ is metric regular. Thus, the claimed result follows by Corollary \ref{cor:bias_free_bnd_ecc}.
\end{proof}

\subsubsection{Adaptive Independent Metropolis with normalizing flows}\label{subsec:normalising_flow_AMH}

We investigate an AIR version of the independent Metropolis-Hastings (IMH) algorithm. 

As in \cite{brofos2022adaptation,gabrie2022adaptive} we adapt the proposal by the use of normalizing flows, 
however, also other strategies are feasible, see e.g.~\cite{parno2018transport}.

Let $\cX \subseteq \R^d$ be compact, $\mathcal{B}(\cX)$ be the corresponding Borel $\sigma$-algebra and $p \colon \cX \to (0, \infty)$ be a {strictly positive} and continuous probability density w.r.t.~the Lebesgue measure, which induces $\pi$. As in Example \ref{ex: norm_flow}, 
let $\mathcal{Z}$ be a set of normalizing flows, parametrized by $u \in \mathcal{U}$.
Let $\widetilde{\pi}$ be another distribution on $(\cX, \mathcal{B}(\cX))$, with density $\widetilde{p}$, such that sampling $\widetilde{X} \sim \widetilde{\pi}$ is feasible. 
Given $X_0= x_0, \dots, X_n=x_n$ within IMH one first proposes a candidate $\widetilde{X} \sim \widetilde{\pi}$, realized by $\widetilde{x}$, which is then accepted as $x_{n+1}$ with probability  
	\[
	a (x_n , \widetilde{x}) = \min\left\{ 1, \frac{p(\widetilde{x}) \, \widetilde{p}(x_n)}{p({x_n}) \,\widetilde{p}(\widetilde{x})} \right\},
	\]
	and
otherwise rejected, such that $x_{n+1} = x_n$.	

It is known from \cite{MeTw96,wang2022exact,brown2024exact} that the `efficiency' of IMH crucially depends on the ratio $w(x) = p(x)/ \widetilde{p}(x)$. 
	Indeed, let
$w^* = \sup_{x \in \cX} w(x) < \infty$, then, the total variation distance between the distribution of $X_n$ and $\pi$ is bounded by $2 (1- 1/w^*)^n$, see for example \cite[Theorem~2.1]{MeTw96},
which renders the IMH Markov chain 
$(X_n)_{n \in \N_0}$ to be uniformly ergodic. 
If $z_u$, with $u \in \mathcal{U}$, is a normalizing flow, then by the change-of-variables formula $z_u(\widetilde{X})$, with $\widetilde{X}\sim \widetilde{\pi}$, has density 
\[
p_{z_u}(x) = 
\widetilde{p}(z_u^{-1}(x)) \vert \det D z_u^{-1}(x) \vert,
\]
where $Dz_{u}^{-1}$ denotes the Jacobi matrix of $z_u^{-1}$.

Hence, for a suitably chosen flow $z_{u}$, proposing $z_{u}({\widetilde{X}})$ may improve the speed of convergence if $w_u^* < w^*$, with $w_{u}^* = \sup_{x \in \cX}p(x)/p_{z_{u}}(x)$.
Intuitively, if $\pi$ and $\pi_{z_u}$ are `close to each other', then $w_u^*$ should be `close to one'.
Thus, we may try to minimize the KL divergence 
over $z \in \mathcal{Z}$ (or $u \in \mathcal{U}$), that is, we try to find a flow $z_u$, in terms of $u\in\mathcal{U}$, that minimizes, 
\[
\KL(\pi,\pi_{z_u}) 
= \int_{\cX} \log\left(\frac{p(x)}{ p_{z_u}(x)}\right) \pi(\d x)
= 
C_\pi - \int_{\cX} \log p_{z_u}(x)\,  \pi(\d x),
\]
where $C_\pi = \int_{\cX} \log p(x) \pi(\d x)$ is a constant independent of $z_u$. 
Also other loss functions are possible; for details we refer to the different settings considered in  \cite{albergo2019flow,nicoli2021estimation,gabrie2022adaptive, grenioux2023transportMaps}.
The integral $\int_{\cX} \log p_{z_u}(x) \pi(\d x)$ is usually intractable, therefore, we rely on the idea to minimize the empirical proxy 
\[
J_{z_u}^{(n)}(Y_1, \dots, Y_n) = -\frac{1}{n} \sum_{j=1}^n \log p_{z_u}(Y_j),
\]
where $Y_1, \dots , Y_n$ is a finite sequence of random variables that approximately samples $\pi$.
Inspired by \cite{gabrie2022adaptive}, in the context of Algorithm \ref{alg:generic_scheme}, we use the already known $X_1, \dots, X_n$ as samples in $J^{(n)}_{z_u}$. This results in the following AIR scheme. To ease notation, we omit the dependence on $u$ in the flows $z_u$. 

\begin{alg}\label{alg:Normalizing_flow_AIR_independent_MH}
	For $n \in\N$	given $(X_0,\dots,X_n)= (x_0, \dots , x_n) \in \cX^{n+1}$ and $(Z_0,\dots,Z_{n-1})= (z_0 , \dots , z_{n-1}) \in \mathcal{Z}^n$ we sample $X_{n+1}$ by performing:
	\begin{enumerate}
		\item \textbf{If} $n = j(j+1)/2$ for some $j \in \N$,
		
		\quad \textbf{then} (approximately) compute\footnote{For simplicity, we assume that this minimum always exists.} $z_n= \min_{z \in \mathcal{Z}} J^{(n)}_z(x_1, \dots, x_n)$; 
		
		\textbf{Else} set $z_{n} = z_{n-1}$.
		\item Sample $\widetilde{X} \sim \widetilde{\pi}$, independent of everything else, with realization $\widetilde{x}$.
		\item Set $x_{n+1} = z_n(\widetilde{x})$ with probability 
		$a_{z_n}(x_n, \widetilde{x}) 
		= \min\left\{ 1, \frac{p(\widetilde{x}) \, {p_{z_n}}(x)}{p({x}) \,{p_{z_n}}(\widetilde{x})} \right\}$, 
		else set $x_{n+1} = x_n$.
	\end{enumerate}
\end{alg}

To analyze Algorithm \ref{alg:Normalizing_flow_AIR_independent_MH}, we impose the following condition.

\begin{ass}\label{ass:IMH_assumption}
Let the proposal density $\widetilde{p}\colon \mathcal{X} \to (0,\infty)$ be strictly positive and continuous. Assume that
	$\mathcal{Z}$ is smoothly parametrized by $\mathcal{U} \subseteq \R^s$, that is, for $z_u\in \mathcal{Z}$ we have that $(u,x) \mapsto z_u(x)$ is continuous.
	Additionally, there exists $\lambda >0$, such that for any $u \in \mathcal{U} $ and $x_0 \in \cX$ we have $\vert \det D_x (z_u)^{-1} (x_0) \vert > \lambda$.
\end{ass}
\begin{prop}\label{thm:AIR_IMH_MSE_bound}
Let $(X_n, Z_n)_{n\in\mathbb{N}_0}$ be determined by Algorithm~\ref{alg:Normalizing_flow_AIR_independent_MH}.
	If Assumption \ref{ass:IMH_assumption} is true, then there exists $\widetilde{C}  \in (0, \infty)$, such that for any bounded measurable $h \colon \cX \to \R$ and $n\in \N$ we have
	\[
	\E \left[ \vert \widehat{\pi}_n (h) - \pi(h)  \vert^2   \right] \leq \frac{\widetilde{C} \norm{h}_{\infty}^2}{n},
	\]
	where $\norm{h}_\infty = \sup_{x \in \cX} \vert h(x) \vert$.
\end{prop}
\begin{proof}
	The aim is to apply Corollary~\ref{cor:bias_free_bnd_ecc}, such that we check all corresponding requirements.
For $z\in\mathcal{Z}$, note that $K_z$ corresponds to an IMH transition kernel with proposal density $p_z$, that is, 
	\[
	K_z(x, A ) = \int_A a_z(x, \widetilde{x}) p_z( \widetilde{x}) \,\d \widetilde{x}+ \mathds{1}_{A}(x) \int_{\cX} \left( 1 -a_z(x,\widetilde{x})\right)  p_z(\widetilde{x}) \,\d \widetilde{x},
	\]
	where $a_z(x,\widetilde{x}) = \min\left\{ 1, \frac{p({\widetilde{x}}) \, {p_z}(x)}{p({x}) \,{p_z}({\widetilde{x}})} \right\}$.
	It is well-known that for any $z \in \mathcal{Z}$ the invariant distribution of $K_z$ is $\pi$.

	Due to the continuity of $p$ and the compactness of $\cX$, there exists some $c_1 \in (0, \infty)$ such that $\sup_{y \in \cX} p(y) = c_1$. 
	By Assumption \ref{ass:IMH_assumption}, there exists some $c_2 \in (0,\infty)$ such that  $ \inf_{y\in\cX}\widetilde{p}(z^{-1}(y)) \vert \det D_x z^{-1}(y) \vert \geq c_2 \lambda > 0$.
	Hence  
	\[
	w_z(y) = \frac{p(y)}{p_z(y)} \leq \frac{\sup_{y \in \cX} p(y)}{\inf_{y\in\cX}\widetilde{p}(z^{-1}(y)) \vert \det D_x z^{-1}(y) \vert} \leq \frac{c_1}{c_2\lambda}< \infty,
	\]
	uniformly for all $z \in \mathcal{Z}$ and $y \in \cX$. 
	From \cite[Theorem~1]{wang2022exact} or \cite[Theorem~2.1]{MeTw96} it follows that the family of kernels $(K_z)_{z \in \mathcal{Z}}$ is simultaneously (in $z\in\mathcal{Z}$) uniformly ergodic.
	Hence, considering the trivial metric\footnote{Recall that this means $\rho(x,y) = 1$ if $x \neq y$ and $0$ else.} within the Wasserstein distance, by the same arguments as in \cite[Section 4.1]{hofstadler2026convergenceRates} it follows, that the simultaneous Wasserstein contraction assumption is satisfied. 
	Moreover, by the fact that the trivial metric is bounded, we also have that $(X_n, Z_n)_{n \in \N_0}$ is metric regular, cf. Remark~\ref{rem:metric_bnd}. Thus, the claimed result follows by Corollary \ref{cor:bias_free_bnd_ecc}.
\end{proof}

\section{Doubly intractable problems}\label{sec:doubly_intractable}
In this section we study the doubly intractable setting of Example \ref{ex: doubly} for a generic simultaneously Wasserstein contractive family $(K_z)_{z \in \mathcal{Z}}$. To provide a fair assessment of the resulting estimates we also add a cost analysis.

We have a probability space $(\mathcal{D},\mathcal{F}_{\mathcal{D}},\mu)$, fixed data $\bar{y}\in\mathcal{D}$ and a reference (prior) distribution $\nu$ on the `parameter space' $(\cX,\mathcal{B}(\cX))$. 
A likelihood function $y\mapsto \ell(y|x)$ is determined by a mapping $E\colon \mathcal{D}\times \cX \to \mathbb{R}$ as
\[
\ell(y|x) = \frac{\exp(-E(y,x))}{z^*(x)},
\]
with $z^*(x) := \int_{\cD} \exp(-E(y,x)) \, \mu({\rm d}y) \in (0,\infty)$ being intractable in the sense of a non-evaluable function $x\mapsto z^*(x)$. The mapping $y\mapsto\ell(y|x)$ fits into the framework of energy based models, cf.~\cite{lecun2006tutorial}. 

There is a rich literature on learning of the parameters of such models by optimization methods such as contrastive divergence, maximum pseudo-likelihood or score matching, see e.g.~\cite{Nijkamp,song2021train}. Assessing uncertainty requires sampling of a posterior distribution $\pi=\pi_{z^*}$ given by
\[
\pi_{z^*}({\rm d}x) = \frac{\ell(\bar{y}| x) \nu({\rm d}x)}{\int_{\cX} \ell(\bar{y}| \widetilde{x}) \nu({\rm d} \widetilde{x})}
=
\frac{\exp(-E(\bar{y},x))\nu({\rm d}x)}{z^*(x)\, C_{z^*}},
\] 
with unknown normalizing constant $C_{z^*} := \int_{\cX} \exp(-E(\bar{y},x))\, z^*(x)^{-1} \nu({\rm d}x)$. Additionally, $x\mapsto z^*(x)$ cannot be evaluated, such that the sampling problem is doubly-intractable.
Examples occur e.g.~in statistical mechanics \cite{habeck2014bayesian}, 
molecular dynamics
\cite{eltzner2023Bayesian}, or exponential random graph models \cite{hunter2006inference}.

To employ classical MCMC algorithms, evaluating $x\mapsto \ell(\bar{y}|x)$, which requires $x\mapsto z^*(x)$, is essential, yet not possible in the present setting. To deal with this issue one may use, for example, augmented state space methods, as considered here \cite{moller2006efficient,murray2012mcmc}, noisy MCMC, cf. \cite{alquier2016noisy,habeck2020stability}, or adaptive MCMC as in \cite{atchade2013bayesian,liang2016adaptive}, see also the survey \cite{park2018bayesian}. 
We consider an adaptive MCMC approach approximating $x\mapsto z^*(x)$, cf.~\cite{habeck2014bayesian,eltzner2023Bayesian}. To define $\mathcal{Z}$ we impose a standing assumption.\\[-1ex]

\noindent
\textbf{Standing assumption.} 
Let $\rho\colon \cX^2 \to [0,\infty)$ be the metric on $\cX$ that renders $(\cX,\rho)$ a Polish space. For $E\colon \mathcal{D}\times \cX \to \mathbb{R}$ we assume for $y\in\mathcal{D}$ that the mapping $x\mapsto\exp(-E(y,x))$ satisfies the following continuity condition: For all $\varepsilon>0$ and for all $x,x'\in\cX$ there exists a $\delta>0$ such that
\[
\rho(x,x') <\delta \; \Longrightarrow \; \sup_{y\in\mathcal{D}} \vert \exp(-E(y,x))-\exp(-E(y,x'))\vert <\varepsilon. 
\]

Let $\mathcal{P}$ be a set of probability measures on $(\mathcal{D},\mathcal{F}_{\mathcal{D}})$ with $\xi=\frac{1}{N}\sum_{j=1}^{N} \delta_{y_j} \in \mathcal{P}$ for arbitrary $N\in\mathbb{N}$ and $y_1,\dots,y_N\in\mathcal{D}$.
We define
\[
\mathcal{Z} = \left\{ z\colon\mathcal{X} \to (0,\infty) \;\bigg \vert \; x\mapsto z(x) = \int_{\mathcal{D}} \exp(-E(y,x))\, \gamma({\rm d} y), \gamma\in\mathcal{P} \right\}.
\]
The standing assumption yields that $\mathcal{Z}$ contains only continuous functions, such that we can equip this set easily with a suitable $\sigma$-algebra $\mathcal{F}_{\mathcal{Z}}$. Now let $(\pi_z)_{z\in\mathcal{Z}}$ be the sequence of probability measures on $(\cX,\mathcal{B}(\cX))$ given by
\[
\pi_z(\d x) = \frac{\exp(-E(\bar{y},x))}{C_{z}\, z(x)} \nu(\d x), \quad z\in\mathcal{Z},
\]
with $C_{z} = \int_\cX \exp(-E(\bar{y},x))z(x)^{-1} \nu(\d x)$.

To use a point-wise proxy for $z^*$ below we rely on the following. 
\begin{ass} \label{ass: u_and_ell}
	Let $\ell,u\colon\cX \to (0,\infty)$ be functions satisfying $\ell(x) \leq \exp(-E(y,x)) \leq u(x)$ for all $(y,x)\in\mathcal{D}\times\cX$. 
	For $x_0\in\mathcal{X}$ let
	\[
	R_1(x_0)
	:=
	\frac{\int_{\mathcal{X}}  \rho(x,x_0)^2\,\exp(-E(\bar{y},x)\ell(x)^{-1}\nu({\rm d}x)}{\int_{\mathcal{X}} \exp(-E(\bar{y},x')) u(x')^{-1} \nu({\rm d}x')}
	\]
	and assume that
	
	\[
	R^{(\ell,u)}:= \inf_{x_0\in\cX}\left[\int_{\cX} (\rho(x_0,x)^2+R_1(x_0)) \frac{u(x)^2}{\ell(x)^2} \pi_{z^*}(\d x)\right]<\infty.
	\]
\end{ass}	
Using Assumption \ref{ass: u_and_ell}, following \cite[Section~3.1]{habeck2020stability}, we obtain the following for an iid Monte Carlo estimator of $z^*$; see Appendix \ref{subsec:appendix_proofs_doubly_intractable} for a proof.
\begin{prop}
	\label{prop: expl_err_bnd}
	Let Assumption~\ref{ass: u_and_ell} be satisfied.
	Then, 
	\begin{equation}
		\mathbb{E}[ W(\pi_{\widehat{Z}^{(N)}},\pi_{z^*})^2] 
		\leq 
		\frac{2R^{(\ell,u)}}{N},
	\end{equation}
	where for $N\in\mathbb{N}$ the $\mathcal{Z}$-valued random variable $\widehat{Z}^{(N)}$ is given as
	$
	x\mapsto\widehat{Z}^{(N)}(x) := \frac{1}{N} \sum_{j=1}^N \exp(-E(Y_j,x))
	$  
	for an iid 
	sequence of $\mu$-distributed random variables $(Y_j)_{j \in \N}$.
\end{prop}

Next we provide the concrete AIR scheme in the doubly intractable setting. We emphasize here that 
$K\colon (\mathcal{Z}\times \cX) \times \mathcal{B}(\cX)  \to [0,1]$, also denoted as $K_z(x,A) = K(z,x,A) $ for $(x,A)\in\cX\times\mathcal{B}(\cX)$ and $z\in\mathcal{Z}$, is not further specified. Partially, this is because we focus 
on the contribution of the error by approximating $z^*$, but also find this generic framework useful for a proof of concept verification. Algorithmically, (under appropriate assumptions) one may have a Metropolis-Hastings scenario in mind, but also Gibbs or slice sampling kernels may be considered. Finally, we also use an increasing function $N\colon\{0,t_1,t_2,\dots\} \to \mathbb{N}\cup\{0\}$, where $t_j = j(j+1)/2$ in the upcoming algorithm.

\begin{alg}
	\label{alg:adaptive_mcmc_doubly_intractable}
	For $n\in\mathbb{N}$ given $(X_0,\dots,X_n)=\widetilde{x}\in\cX^{n+1}$ with $X_n=x_n\in\cX$ and $(Z_0,\dots,Z_{n-1})=\widetilde{z}\in\cZ^{n}$ we sample $X_{n+1}$ by performing:
	\begin{enumerate}
		\item \textbf{If} $n= j(j+1)/2$ for some $j \in \N$,
		
		\quad \textbf{then} sample $Y_{N(t_{j-1})+1},\dots,Y_{N(t_j)}$ iid w.r.t.~$\mu$, call the result
		
		\quad $y_{N(t_{j-1})+1},\dots,y_{N(t_j)}$ and set $x\mapsto z_n(x)$ as 
		\begin{align*}
			z_n(x) 
			&:= \frac{1}{N(n)} \sum_{k=1}^{N(n)} \exp(-E(y_k,x)).
		\end{align*}
		
		\textbf{Else} set $z_{n} = z_{n-1}$, i.e., $Z_n\sim \delta_{z_{n-1}}$.
		\item Sample $X_{n+1}\sim K(z_{n},x_{n},\cdot)$.
	\end{enumerate}
\end{alg}

Observe that $z_n$ from Algorithm~\ref{alg:adaptive_mcmc_doubly_intractable}
coincides with a realization of $\widehat{Z}^{(N(t_{m(n)}))}$ from Proposition~\ref{prop: expl_err_bnd}. Next, under appropriate assumptions we provide an error bound for $(X_n)_{n \in \N_0}$ corresponding to Algorithm \ref{alg:adaptive_mcmc_doubly_intractable}.
Using Corollary \ref{cor:error_decomp_bounded_eccentricities} together with Proposition \ref{prop: expl_err_bnd} we obtain the following. 

	\begin{prop}
	\label{prop: doubly_intract_err_bnd}	
Let $(K_z)_{z \in \mathcal{Z}}$ in the setting of this subsection be simultaneously Wasserstein contractive and Assumption~\ref{ass: u_and_ell} be true. Let $(X_n, Z_n)_{n\in\mathbb{N}_0}$ be specified via Algorithm~\ref{alg:adaptive_mcmc_doubly_intractable} with $N(t_j) = 
	\left\lceil
	 j^{1+\alpha}
	\right\rceil 
	$ for $j\in\mathbb{N}$ and $\alpha\geq0$, and assume $(X_n, Z_n)_{n \in \N_0}$ is metric regular. Then, there exists $\widetilde{C}\in (0,\infty)$ (depending on $\alpha, \kappa, \Lambda, \tau, C, R^{(\ell,u)}$) such that
	\[
	\mathbb{E}\vert \widehat{\pi}_n(h) - \pi_{z^*}(h) \vert^2 
	\leq 3 \widetilde{C} \Vert h \Vert_{\rho}^2 \cdot
	\begin{cases}
		{n^{-\frac{1+\alpha}{2}}(1-\alpha)^{-1}} & \alpha\in [0,1)\\
		{n^{-1}} \,{\log(n)} & \alpha=1\\
		{n^{-1}(1-\alpha)^{-1}} & \alpha\in (1,\infty),	
	\end{cases}
	\]
	for any $n \in \N$ and Lipschitz function $h \colon \cX \to \R$ such that $\pi_z(h) < \infty$ for any $z \in \mathcal{Z} \cup \{z^*\}$.
\end{prop}
Now, let us study Proposition \ref{prop: doubly_intract_err_bnd} within the following \textbf{cost model}:
\begin{enumerate}[(c1)]
	\item A single evaluation of $E\colon \mathcal{D} \times \mathcal{X} \to \mathbb{R}$ costs $c_1>0$; 
	\item a single sample generation w.r.t.~$\mu$ costs $c_2>0$; and
	\item a single evaluation of $h\colon\mathcal{X}\to \mathbb{R}$ costs $c_3>0$.
\end{enumerate}
Under this model one evaluation of $x \mapsto z_i(x)$ `costs' $N(t_{m(i)}) c_1$, with $i \in \N$ and $m(i)$ as in \eqref{equ:m(n)_definition}. 
Standard MCMC algorithms, e.g.~Metropolis-Hastings, typically require one evaluation of $x \mapsto \exp(-E(\bar{y},x))/z_k(x)$ in the $k$-th iteration. 
Hence, after $n\in\mathbb{N}$ iterations of Algorithm~\ref{alg:adaptive_mcmc_doubly_intractable} we assume a total cost of
\begin{align*}
	\textrm{cost}(\widehat{\pi}_n):=c_1 \sum_{i=1}^n N(t_{m(i) }) + c_2 N(t_{m(n)})+ (c_1+c_3) n
\end{align*}
is involved.
We have the following auxiliary result, proven in the appendix.
\begin{lem}\label{lem:cost_doubly_intractable}
	Within the setting of Proposition~\ref{prop: doubly_intract_err_bnd} we have\footnote{For non-negative $(a_r)_{r\in I}$ and  $(b_r)_{r\in I}$ with arbitrary index set $I\not=\emptyset$, we write $a_r\preceq b_r$ if there exists $c\in\mathbb{R}$, such that $a_r\leq c b_r$ 
		for all $r\in I$ and $a_r\asymp b_r$ if $a_r \preceq b_r$ and $b_r \preceq a_r$.}
	\begin{align*}
		N(t_{m(n)}) &\asymp  n^{\frac{1+\alpha}{2}} \quad\text{and} \quad
		\sum_{i=1}^{n}N(t_{m(i)}) \asymp n^{\frac{3+\alpha}{2}}.
	\end{align*}
\end{lem}
If $c_1, c_2, c_3$ are closely comparable, e.g., $c_1=c_2=c_3$, then $\textrm{cost}(\widehat{\pi}_n) \asymp n^{\frac{3+\alpha}{2}}$ and we can set this into relation of the mean squared error.

\begin{cor}
	Under the assumptions of Proposition~\ref{prop: doubly_intract_err_bnd}, with  ${\rm cost}(\widehat{\pi}_n) \asymp n^{\frac{3+\alpha}{2}}$, for $\varepsilon\in (0,1)$
	let
	\[
	{\rm cost}_\varepsilon := \inf\big\{  {\rm cost}(\widehat{\pi}_n) \mid \sup_{\Vert h \Vert_{\rho}\leq 1} \mathbb{E}\vert \widehat{\pi}_n(h) - \pi_{z^*}(h) \vert^2 \leq \varepsilon,\, n\in\mathbb{N} \big\}.
	\]
	Then, for
	\begin{itemize}
		\item 
		$\alpha\in[0,1)$, we have ${\rm cost}_\varepsilon \preceq \varepsilon^{-\frac{3+\alpha}{1+\alpha}}$;
		\item $\alpha=1$, we have ${\rm cost}_\varepsilon \preceq
		\Big ( \frac{\log(\varepsilon^{-1})}{\varepsilon} \Big)^{2} $;
		\item $\alpha\in (1,\infty)$, we have  ${\rm cost}_\varepsilon \preceq
		\varepsilon^{-\frac{3+\alpha}{2}}$. 
	\end{itemize}
\end{cor}
\begin{proof}
	Consider the first case, that is, $\alpha \in [0,1)$. For treating it, set $n_0 = \big\lceil \big( \frac{3\widetilde{C}}{\eps}\big)^{2/(1+\alpha)}  \big\rceil$. 
	By Proposition~\ref{prop: doubly_intract_err_bnd}, we obtain $\sup_{\norm{h}_{\rho}\leq 1}\E[\vert \widehat{\pi}_{n_0}(h) - \pi(h)\vert^2 ] \leq \eps$. Consequently, since $\textrm{cost}(\widehat{\pi}_n) \asymp n^{\frac{3+\alpha}{2}}$, 
	\[
	\text{cost}_\varepsilon \preceq n_0^{\frac{3+\alpha}{2}} \preceq \varepsilon^{-\frac{3+\alpha}{1+\alpha}}.
	\]
	
	Consider the second case, that is, $\alpha=1$. Set $n_0 :=\lceil \frac{a}{\varepsilon} \log\big( \frac{a}{\varepsilon}\big) \rceil$ with $a=\widetilde{C}(1+e^{-1})$ we have $\widetilde{C} \log(n_0)/n_0 \leq \varepsilon$, see for example \cite[below Remark~2.2]{Hinrichs}. Taking $\textrm{cost}(\widehat{\pi}_n) \asymp n^{\frac{3+\alpha}{2}}$ and the error bound of Proposition~\ref{prop: doubly_intract_err_bnd} into account yields the desired statement.
	
	Consider the third case, that is, $\alpha \in (1, \infty)$. Set $n_0 = \big\lceil \frac{3\widetilde{C}}{\eps} \big\rceil$, such that by Proposition~\ref{prop: doubly_intract_err_bnd}, we obtain (as before) that
		$
			\text{cost}_\varepsilon \preceq n_0^{\frac{3+\alpha}{2}} \preceq \eps^{-\frac{3+\alpha}{2}}.
		$
\end{proof}

\section*{Acknowledgements}
JH is grateful for financial support from the EPSRC grant EP/W026899/2 MaThRad. Furthermore, JH and DR gratefully acknowledge the support of the German Research Foundation (DFG) within project 432680300 -- SFB 1456
subprojects B02. Finally, DR thanks the DFG for their support within the projects with grant numbers 522337282 and 578618598.

\begin{appendix}

\section{Proof of Theorem \ref{thm:mse_error_general}}\label{app:proof_main_results}
We begin with an auxiliary lemma, which allows to split the error into a martingale, an  adaptation and a bias part. This idea is not new and in one way or another has been used around the Poisson's equation approach for adaptive MCMC, see also \cite{andrieu2006ergodicity,saksman2010ergodicity,atchade2010cautionary,fort2011convergence,atchade2013bayesian,laitinen2024invitation, hofstadler2026convergenceRates}.

Throughout the section write $t_{k} = k(k+1)/2$ for $k \in \N$. Note that these numbers correspond exactly to the adaptation times in Algorithm \ref{alg:generic_scheme}.

\begin{lem}\label{lem:error_decomposition}
	Let the assumptions, the setting and $h\colon \cX \to \R$ be as in Theorem~\ref{thm:mse_error_general}. Then, for $n\in\mathbb{N}$ we have
	\begin{equation}
		\label{al:MSE_repr} 
		\mathbb{E}\vert \widehat{\pi}_n(h) - \pi_{z^*}(h) \vert^2 
		\leq \frac{3}{n^2} \left[ \mathbb{E}(M_n^2) + \mathbb{E}(A_n^2) + \mathbb{E}(B_n^2) \right],
	\end{equation}
	where we define the martingale, adaptation and bias terms respectively as
	\begin{align*}
		M_n &:= \sum_{j=1}^n  (u_{Z_{j}}(X_{j+1})-K_{Z_{j}}u_{Z_{j}}(X_{j})),\\
		A_n & := \sum_{j=1}^{m-1} ( u_{Z_{t_{j}}} (X_{t_{j}}) -   u_{Z_{t_j}} (X_{t_{j+1}})) 
		+ u_{Z_{t_m}}(X_{t_m}) - u_{Z_{t_m}}(X_{n+1}),\\
		B_n & := \sum_{j=1}^n ( \pi_{Z_{j}}(h)-\pi_{z^*}(h) ),
	\end{align*}
	and $u_z$ solves Poisson's equation with forcing function $h$, cf.~Proposition \ref{prop:solution_poisson_equation}.
\end{lem}
\begin{proof}
	If not stated otherwise, all steps of the proofs are understood $\mathbb{P}$-almost surely.
	The definition of $B_n$ yields
	\begin{align}\label{equ:martingale_decomposition_1}
		\sum_{j=1}^n (h(X_j)-\pi_{z^*}(h))  
		&= \sum_{j=1}^n \left( h(X_j) - \pi_{Z_j}(h) \right)   + B_n .
	\end{align}
	Since $\mathbb{E}[{\ecc_{Z_j}}(X_k)] < \infty$, the eccentricities can be infinite at most on a set of $\mathbb{P}$-measure zero. Hence, by Proposition~\ref{prop:solution_poisson_equation}, 
	\begin{align*}
		h(X_j) - \pi_{Z_j}(h) &= u_{Z_j} (X_j) - K_{Z_j} u_{Z_j} (X_j) \\ &= u_{Z_j} (X_{j+1}) - K_{Z_j} u_{Z_j} (X_j) + u_{Z_j} (X_{j}) - u_{Z_j} (X_{j+1}).
	\end{align*}
	Consequently, 
	\begin{equation}\label{equ:martingale_decomposition_2}
		\sum_{j=1}^n \left( h(X_j) - \pi_{Z_j}(h) \right) = M_n + \sum_{j=1}^n (u_{Z_j} (X_{j}) - u_{Z_j} (X_{j+1})).
	\end{equation}
	Fix $k \in \N$, such that for
	$j \in \{t_{k}, \dots, t_{k+1}-1  \}$ it follows from the specification of Algorithm \ref{alg:generic_scheme} that $Z_j \equiv Z_{t_k}$. 
	Thus, a telescoping argument yields
	\begin{align*}
		& \sum_{j=t_k}^{t_{k+1}-1} (u_{Z_j} (X_{j}) - u_{Z_j} (X_{j+1})) = u_{Z_{t_k}}(X_{t_k}) - u_{Z_{t_{k}}}(X_{t_{k+1}}),\\
		& \sum_{j=t_m}^{n} (u_{Z_j} (X_{j}) - u_{Z_j} (X_{j+1})) = u_{Z_{t_m}}(X_{t_m})-u_{Z_{t_m}}(X_{n+1}),
	\end{align*}
	where the last equation make sense, since $m=m(n)$ satisfies $t_m \leq n <t_{m+1}$. By the latter identities and the fact that $t_1=1$ we see
	\begin{align}\label{equ:martingale_decomposition_3}
		&\sum_{j=1}^n  (u_{Z_j} (X_{j}) - u_{Z_j} (X_{j+1})) \\ &= \sum_{k=1}^{m-1} \sum_{j=t_k}^{t_{k+1}-1} (u_{Z_j} (X_{j}) - u_{Z_j} (X_{j+1})) + \sum_{j=t_m}^{n} (u_{Z_j} (X_{j}) - u_{Z_j} (\xi_{j+1})) \nonumber =  A_n.
	\end{align}
	Combining \eqref{equ:martingale_decomposition_1},\eqref{equ:martingale_decomposition_2} and \eqref{equ:martingale_decomposition_3} we get 
	\[
	\sum_{j=1}^n (h(X_j)-\pi_{z^*}(h))  = M_n + A_n + B_n.
	\]
	By applying the inequality $(a+b+c)^2 \leq 3 (a^2 + b^2 + c^2)$, that is true for any $a,b,c \in \R$, we finish the proof. 
\end{proof}

The upcoming lemma shows that $(M_n)_{n \in \N}$ is indeed a martingale. 

\begin{lem}
	\label{lem: martingale}
	Let the assumptions and the setting be as in Theorem~\ref{thm:mse_error_general}.	 Then, the sequence $(M_n)_{n \in \N}$ defined in Lemma \ref{lem:error_decomposition} is an $(\mathcal{F}_n)_{n\in\mathbb{N}}$-martingale, where $\mathcal{F}_n:=\sigma(X_0,\dots,X_n,Z_0,\dots,Z_n)$.
\end{lem}
\begin{proof}
	As in the proof of Lemma \ref{lem:error_decomposition} we note that all requirements to apply Proposition~\ref{prop:solution_poisson_equation} are met. Hence,	
	\begin{align*}
		& \mathbb{E}\vert M_n \vert 
		\leq  \sum_{k=1}^n \mathbb{E} \vert u_{Z_{k}}(X_{k+1})-K_{Z_{k}}u_{Z_{k}}(X_{k})\vert \\
		&\leq 
		\sum_{k=1}^n 
		\mathbb{E} \left[ \int_\cX \vert u_{Z_{k}}(X_{k+1}) - u_{Z_{k}}(x) \vert  K(Z_{k},X_{k},{\rm d}x)  \right]\\
		&  	\leq {\Gamma \Vert h \Vert_{\rho} } \sum_{k=1}^n \mathbb{E}\left[\int_\cX \rho(X_{k+1},x) K(Z_{k},X_{k},{\rm d}x) \right]\\
		&  =  {\Gamma \Vert h \Vert_{\rho} } \sum_{k=1}^n \mathbb{E}\left[\int_\cX\int_{\cX}\rho(\overline{x},x) K(Z_{k},X_{k},{\rm d}x) K(Z_{k},X_{k},{\rm d}\overline{x}) \right]\\
		&\leq  {n \Gamma \Vert h \Vert_{\rho} \sqrt{\Lambda}}  <\infty.
	\end{align*}
	Furthermore,
	\begin{align*}
		\mathbb{E}[M_{n+1}\mid \mathcal{F}_n]
		= M_n + \mathbb{E}(u_{Z_{n}}(X_{n+1})\mid X_n,Z_n)-K_{Z_n}u_{Z_n}(X_n) = M_n,
	\end{align*}
	where we used that $\mathbb{E}(u_{Z_{n}}(X_{n+1})\mid X_n,Z_n)=K_{Z_n}u_{Z_n}(X_n)$, which is justified by the disintegration theorem, cf.~\cite[Theorem~6.4]{kallenberg1997foundations}.
\end{proof}

Now we are able to prove the main result, i.e.~Theorem \ref{thm:mse_error_general}.

\begin{proof}[Proof of Theorem~\ref{thm:mse_error_general}]
	By Lemma~\ref{lem:error_decomposition} it is sufficient to estimate  $\mathbb{E}(M_n^2)$, $\mathbb{E}(A_n^2)$ and $\mathbb{E}(B_n^2)$ suitably.
	
	\textbf{We start with the $\mathbb{E}(M_n^2)$ term:} 
	Set $\Delta_k = u_{Z_{k}}(X_{k+1})-K_{Z_{k}}u_{Z_{k}}(X_{k})$, such that $M_n = \sum_{j=1}^n \Delta_j$. 
	By Lemma \ref{lem: martingale} $(M_n)_{n \in \N}$ is a martingale, such that $
	\mathbb{E}(\Delta_k \Delta_j)  = 0$ for $j,k\in\mathbb{N}$ with $j\neq k$.
	Consequently,
	\begin{align*}
		\mathbb{E}(M_n^2) = \sum_{k=1}^{n}\sum_{j=1}^{n} \mathbb{E}(\Delta_k \Delta_j) = \sum_{j=1}^n \mathbb{E} \Delta_j^2.
	\end{align*}
	By Proposition \ref{prop:solution_poisson_equation} solutions of Poisson's equation are Lipschitz, such that
	\begin{align*}
		\vert \Delta_k \vert 
		=  & \vert u_{Z_{k}}(X_{k+1})-K_{Z_{k}} 		
		u_{Z_{k}}(X_{k}) \vert 
		\leq  \int_\cX \vert u_{Z_{k}}(X_{k+1}) - u_{Z_{k}}(x) \vert K_{Z_{k}}(X_{k},{\rm d}x)\\
		\leq& {\Gamma\Vert h \Vert_{\rho}} \int_\cX \rho(X_{k+1},x) K_{Z_{k}}(X_{k,}{\rm d}x).
	\end{align*}
	Jensen's inequality yields
	\begin{align*}
		\mathbb{E} \Delta_k^2 & \leq {\Gamma^2\Vert h \Vert_{\rho}^2 } \mathbb{E} \Big[ \int_\cX \rho(X_{k+1},x)^2 K_{Z_{k}}(X_{k},{\rm d}x)\Big]\\
		& = 
		{\Gamma^2\Vert h \Vert_{\rho}^2 } \mathbb{E} \Big[ \mathbb{E}\Big[\int_\cX \rho(X_{k+1},x)^2 K_{Z_{k}}(X_{k},{\rm d}x)\mid Z_{k},X_{k}\Big]\Big]\\
		& = 
		{\Gamma^2\Vert h \Vert_{\rho}^2} \mathbb{E} \Big[ \int_\cX \int_\cX \rho(\overline{x},x)^2 K_{Z_{k}}(X_{k},{\rm d}x)K_{Z_{k}}(X_{k},{\rm d}\overline{x})\Big]
		\leq 
		{\Gamma^2\Vert h \Vert_{\rho}^2\Lambda}.
	\end{align*}
	Therefore 
	\[
	\mathbb{E}(M_n^2) \leq {n\Gamma^2 \Vert h \Vert_{\rho}^2\Lambda}.
	\]
	
	\textbf{We turn to the $\mathbb{E}(A_n^2)$ term:} By an application of Jensen's inequality (for sums)
	we obtain
	\[
	A_n^2 \leq m \Big( \sum_{k=1}^{m-1} \left\vert u_{Z_{t_k}} (X_{t_k}) - u_{Z_{t_k}}(X_{t_{k+1}})  \right\vert^2 + \left\vert u_{Z_{t_m}} (X_{t_{m}}) - u_{Z_{t_m}}(X_{n+1}) \right\vert^2  \Big).
	\]
	Due to Proposition \ref{prop:solution_poisson_equation} and the inequality $(a+b)^2 \leq 2(a^2+b^2)$, valid for any $a,b \in \R$, we get for any $k \in \N$ that
	\begin{align*}
		\left\vert u_{Z_{t_k}} (X_{t_k}) - u_{Z_{t_k}}(X_{t_{k+1}})  \right\vert^2 
		&\leq 
		{2(\Gamma\Vert h \Vert_{\rho})^2} \left( 
		{\ecc}_{Z_{t_k} } (X_{t_k})^2 + {\ecc}_{Z_{t_k}} (X_{t_{k+1}})^2 \right).
	\end{align*}
	By the same arguments, we estimate $\left\vert u_{Z_{t_m}} (X_{T_{m}}) - u_{Z_{t_m}}(X_{n+1}) \right\vert^2 $. 
	Combining this with the bound for the expected eccentricities and \eqref{equ:m(n)_estimate}, it follows 
	\[
	\E [A_n^2] \le m^2 2(\Gamma\Vert h \Vert_{\rho})^2 \kappa \leq 4n (\Gamma\Vert h \Vert_{\rho})^2 \kappa.
	\]
	\textbf{We turn to the $\mathbb{E}(B_n^2)$ term:} Observe that by Jensen's inequality and $Z_j=Z_{t_k}$ for $j\in\{t_k,\dots,t_{k+1}-1\}$ with $t_1=1$, we have
	\begin{align}
		\notag
		& \frac{\mathbb{E}(B_n^2)}{n} 
		\leq \sum_{j=1}^{n} \mathbb{E}\vert \pi_{Z_j}(h)-\pi_{z^*}(h) \vert^2
		\leq \sum_{k=1}^m
		\sum_{j=t_k}^{t_{k+1}-1} \mathbb{E}\vert \pi_{Z_j}(h)-\pi_{z^*}(h) \vert^2
		\\	
		\label{al: earlier_stop}
		&= \sum_{k=1}^m
		(t_{k+1}-t_{k})
		\mathbb{E}\vert \pi_{Z_{t_{k}}}(h)-\pi_{z^*}(h) \vert^2
		= \sum_{k=1}^m
		(k+1)
		\mathbb{E}\vert \pi_{Z_{t_{k}}}(h)-\pi_{z^*}(h) \vert^2
		\\
		& 
		\notag
		\leq \Vert h \Vert_\rho^2 \sum_{k=1}^{m}
		(k+1) \sup_{\Vert g \Vert_\rho\leq 1} \mathbb{E}\vert \pi_{Z_{t_k}}(g)-\pi_{z^*}(g) \vert^2
		= \Vert h \Vert_\rho^2 B(m),
	\end{align}
	with $B(m) = \sum_{k=1}^{m}
	(k+1) \sup_{\Vert g \Vert_\rho\leq 1} \mathbb{E}\vert \pi_{Z_{t_k}}(g)-\pi_{z^*}(g) \vert^2$.
	
	Combining these bounds with Lemma \ref{lem:error_decomposition} finishes the proof.

\end{proof}

\section{Proofs of Section \ref{sec:doubly_intractable}}\label{subsec:appendix_proofs_doubly_intractable}

\begin{proof}[Proof of Proposition \ref{prop: expl_err_bnd}]
	For arbitrary $\widehat{z}\in\mathcal{Z}$ and $x_0\in\mathcal{X}$, observe that by the Kantorovich-Rubinstein duality
	\[
	W(\pi_{\widehat{z}},\pi_{z^*}) 
	= 
	\sup_{\Vert h \Vert_{\rho}\leq 1,\; h(x_0)=0}
	\left \vert \pi_{z^*}(h)-\pi_{\widehat{z}}(h) \right \vert.
	\]
	Assume that $h\colon \mathcal{X} \to \mathbb{R}$ satisfies $\Vert h \Vert_{\rho}\leq 1$ and $h(x_0)=0$. Then
	\begin{align*}
		&	\vert \pi_{z^*}(h)-\pi_{\widehat{z}}(h)\vert 
		\leq 	\int_{\mathcal{X}} \rho(x,x_0)
		\left\vert \frac{1}{C_{z^*}z^*(x)} - \frac{1}{C_{\widehat{z}}\,
			\widehat{z}(x)} \right\vert
		\exp(-E(\bar{y},x)) \nu({\rm d}x)\\	  
		&	\leq \int_{\mathcal{X}} \rho(x,x_0) 
		\left \vert \frac{z^*(x)}{\widehat{z}(x)}-1\right \vert \pi_{z^*}(\d x)+\left\vert \frac{C_{\widehat{z}}}{C_{z^*}} -1 \right \vert \int_{\mathcal{X}}  \rho(x,x_0)\, \pi_{\widehat{z}}({\rm d}x).
	\end{align*}
	By Jensen's inequality we obtain
	\begin{align} \notag
		W(\pi_{\widehat{Z}^{(N)}},\pi_{z^*})^2
		\leq & 2 \int_{\mathcal{X}} \rho(x,x_0)^2 \;\Big \vert \frac{z^*(x)}{\widehat{Z}^{(N)}(x)}-1\Big \vert^2 \pi_{z^*}(\d x)\\
		& \qquad +2\;\Big\vert \frac{C_{\widehat{Z}^{(N)}}}{C_{z^*}} -1 \Big \vert^2 \int_{\mathcal{X}}  \rho(x,x_0)^2\, \pi_{\widehat{Z}^{(N)}}({\rm d}x). \label{al: est_before_mean}
	\end{align}
	By Assumption~\ref{ass: u_and_ell} we have $\ell(x)\leq\widehat{Z}^{(N)}(x)\leq u(x)$, such that
	$C_{\widehat{Z}^{(N)}} \geq \int_{\mathcal{X}} \exp(-E(\bar{y},x')) u(x')^{-1} \nu({\rm d}x')$
	and  
	$
	\int_{\mathcal{X}}  \rho(x,x_0)^2\, \pi_{\widehat{Z}^{(N)}}({\rm d}x)
	\leq R_1(x_0).
	$
	Estimating \eqref{al: est_before_mean} by incorporating $R_1(x_0)$ and afterwards taking the expectation yields (with a Fubini argument)
	\begin{align}
		\notag
		\mathbb{E}[	W(\pi_{\widehat{Z}^{(N)}},\pi_{z^*})^2]
		&\leq 2 \int_{\mathcal{X}} \rho(x,x_0)^2 \;\mathbb{E}\left[\Big \vert \frac{z^*(x)}{\widehat{Z}^{(N)}(x)}-1\Big \vert^2\right] \pi_{z^*}(\d x)\\	
		& \qquad\qquad\qquad\qquad\qquad +2R_1(x_0)\,\mathbb{E}\left[\Big\vert \frac{C_{\widehat{Z}^{(N)}}}{C_{z^*}} -1 \Big \vert^2\right].
		\label{al: almost_done}
	\end{align}
	By Jensen's inequality and again by a Fubini argument we have 
	\begin{align}
		\notag
		& \mathbb{E}\left[\Big\vert \frac{C_{\widehat{Z}^{(N)}}}{C_{z^*}} -1 \Big \vert^2\right]
		= \mathbb{E}\left[\Big\vert
		\int_{\mathcal{X}} \Big ( \frac{z^*(x)}{\widehat{Z}^{(N)}(x)}-1 \Big)
		\pi_{z^*}({\rm d}x)
		\Big \vert^2\right]\\
		& \leq \mathbb{E}\left[\int_{\mathcal{X}} \Big ( \frac{z^*(x)}{\widehat{Z}^{(N)}(x)}-1 \Big)^2
		\pi_{z^*}({\rm d}x)
		\right] 
		= \int_{\mathcal{X}} \mathbb{E}\Big \vert \frac{z^*(x)}{\widehat{Z}^{(N)}(x)}-1 \Big\vert^2
		\pi_{z^*}({\rm d}x).
		\label{al: really_almost_done}
	\end{align}
	Using $\ell(x)\leq\widehat{Z}^{(N)}(x)\leq u(x)$ yields
	\begin{align}
		\notag
		\mathbb{E}\Big \vert \frac{z^*(x)}{\widehat{Z}^{(N)}(x)}-1 \Big\vert^2
		& = \mathbb{E}\Big \vert  \frac{z^*(x)}{\widehat{Z}^{(N)}(x)} \Big (1- \frac{\widehat{Z}^{(N)}(x)}{z^*(x)} \Big) \Big\vert^2 
		\leq \frac{z^*(x)^2}{\ell(x)^2} \mathbb{E}\Big \vert   \frac{\widehat{Z}^{(N)}(x)}{z^*(x)}-1  \Big\vert^2.
	\end{align}
	The fact that $\frac{\widehat{Z}^{(N)}(x)}{z^*(x)} = \frac{1}{N}\sum_{i=1}^N\frac{\exp(-E(Y_i,x))}{z^*(x)}$ with $\mathbb{E}(\frac{\exp(-E(Y_1,x))}{z^*(x)})=1$ and the iid property of $Y_1,\dots,Y_N$ imply
	\begin{align*}
		&\mathbb{E}\Big \vert   \frac{\widehat{Z}^{(N)}(x)}{z^*(x)}-1  \Big\vert^2
		= \frac{1}{N} 	\,\mathbb{E}\Big \vert   \frac{\exp(-E(Y_1,x))}{z^*(x)}-1  \Big\vert^2\\
		=& \frac{1}{N} \left( \frac{\mathbb{E}[\exp(-2E(Y_1,x))]}{z^*(x)^2}-1 \right) 
		\leq \frac{1}{N} \; \frac{u(x)^2}{z^*(x)^2},
	\end{align*}
	where the last inequality follows by $\exp(-E(Y_1,x))\leq u(x)$. Therefore,
	$	\mathbb{E}\Big \vert \frac{z^*(x)}{\widehat{Z}^{(N)}(x)}-1 \Big\vert^2 \leq \frac{1}{N} \frac{u(x)^2}{\ell(x)^2}$, such that, by \eqref{al: almost_done}, \eqref{al: really_almost_done} and taking an infimum over $x_0\in\mathcal{X}$, the assertion is proven.
\end{proof}

\begin{proof}[Proof of Lemma \ref{lem:cost_doubly_intractable}]
	Using \eqref{equ:m(n)_estimate} yields $N(t_{m(i)}) \leq \left\lceil (4i)^{\frac{1+\alpha}{2}} \right\rceil \leq (8i)^{\frac{1+\alpha}{2}}$. Hence estimating the sum by its largest summand
	\[
	\sum_{i=1}^{n} N(t_{m(i)}) \leq 8^{\frac{1+\alpha}{2}} \sum_{i=1}^n i^{\frac{1+\alpha}{2}} \leq 8^{\frac{1+\alpha}{2}} n^{\frac{3+\alpha}{2}}.
	\]
	Similarly, for the lower bound we use again \eqref{equ:m(n)_estimate} and obtain
	\[
	\sum_{i=1}^{n} N(t_{m(i)}) \geq \frac{1}{2^{\frac{1+\alpha}{2}}} \sum_{i=1}^n i^{\frac{1+\alpha}{2}} 
	\geq 
	\frac{1}{2^{\frac{1+\alpha}{2}}} \sum_{i=\lfloor n/2 \rfloor}^n i^{\frac{1+\alpha}{2}}
	\geq 
	\frac{1}{2^{\frac{3+\alpha}{2}}} n^{\frac{3+\alpha}{2}}.\qedhere
	\]
\end{proof}

\end{appendix}

\newcommand{\etalchar}[1]{$^{#1}$}
\providecommand{\bysame}{\leavevmode\hbox to3em{\hrulefill}\thinspace}
\providecommand{\MR}{\relax\ifhmode\unskip\space\fi MR }
\providecommand{\MRhref}[2]{%
	\href{http://www.ams.org/mathscinet-getitem?mr=#1}{#2}
}
\providecommand{\href}[2]{#2}

\end{document}